\documentclass[12pt,leqno]{amsart}
\usepackage{anyfontsize}

\usepackage{amsmath, amsthm, amsfonts, amssymb, mathrsfs}
\usepackage{graphicx, color}
\usepackage[bookmarksnumbered, colorlinks, plainpages]{hyperref}
\usepackage{cancel}
\usepackage{comment}
\usepackage{xcolor}
\definecolor{darkgreen}{rgb}{0,0.6,0}

\newcommand{\Chi}{\raise 1.5pt\hbox{$\chi$}}

\newcommand{\beq}{\begin{equation}}
\newcommand{\eeq}[1]{\label{#1}\end{equation}}

\newcommand{\avg}{\raise 0pt\hbox{$-$}\hskip -10.7pt\int}

\newcommand{\restrict}[1]{|_{\raise-2pt\hbox{$\scriptstyle #1$}}}

\newtheorem{lem}{Lemma}[section]

\newtheorem{thm}[lem]{Theorem}

\newtheorem{defn}[lem]{Definition}

\newtheorem{remark}[lem]{Remark}

\numberwithin{equation}{section}

\begin{document}
\fontsize{13.4}{15}\selectfont  

\title[Quasilinear equations for H\"ormander $q$-sub-Laplacians]{Quasilinear equations  with exponential growth for H\"ormander $q$-sub-Laplacians on stratified Lie groups}

\author[J. A. León Tordecilla]{Jesus A. León Tordecilla$^1$}
\author[D. Cardona]{Duv\'an Cardona$^{2,3*}$}

\address{
  $^{1}$Jesus A. León Tordecilla:
  \endgraf
  Área de Ciencias Básicas Exactas, Grupo de investigación Giofop,
  \endgraf
  Universidad del Sinú, Seccional Cartagena, Avenue El bosque Trans. 54 \# 30-72,
  \endgraf
  Cartagena, 130001, Colombia
  \endgraf
  {\itshape E-mail address: jesus.leon@unisinu.edu.co}  \endgraf
  \texttt{}
  \bigskip
  \endgraf
  $^{2,3}$ Duv\'an Cardona:
  \endgraf
  Department of Mathematics: Analysis, Logic and Discrete Mathematics
  \endgraf
  Ghent University,
  \endgraf
  Ghent-Belgium.
  \endgraf
  {\it E-mail address:} {\rm duvanc306@gmail.com, duvan.cardonasanchez@ugent.be.}
  \endgraf
 Department of Mathematics
  \endgraf
  Pontificia Universidad Javeriana, 
  \endgraf
  Bogot\'a-Colombia.
  \endgraf
  {\itshape E-mail addresses:}  
  \endgraf \texttt{duvanc306@gmail.com},  
  \endgraf
  \texttt{duvan.cardonasanchez@ugent.be}, 
  \endgraf
  \texttt{cardonaduvan01@javeriana.edu.co}
}

\thanks{Jesus A. León Tordecilla$^1$ has been supported by Universidad del Sinú, Cartagena-Colombia.   Duv\'an Cardona$^{2,3,4,*}$  has been  supported  by the FWO  Odysseus  1  grant  G.0H94.18N:  Analysis  and  Partial Differential Equations and by the Methusalem programme of the Ghent University Special Research Fund (BOF)
(Grant number 01M01021), by the FWO Fellowship
Grant No 1204824N and by the FWO Grant K183725N of the Belgian Research Foundation FWO. He also has been supported by the Oberwolfach Leibniz Fellow of the Mathematical Institute of Research of Oberwolfach, MFO-Germany, Project F2511 (2026) and by the Department of Mathematics of the Pontificia Universidad Javeriana, Bogot\'a-Colombia.   The authors have been benefited by the support of the Research Group: Analysis and PDE in Developing Countries of ICMAM Latin America (International Community of Mathematicians from Latin America).  }
\date{}

\begin{abstract} We prove the existence of positive weak solutions for a quasilinear  equation with exponential nonlinearity on arbitrary stratified Lie groups for \emph{horizontal $q$-Laplacians,} also called  $q$-sub-Laplacians on these groups. The nonlinearity combines a concave term with an exponential growth that can be subcritical, critical, or supercritical with respect to the Trudinger--Moser inequality for Lorentz spaces on stratified Lie groups. We prove a version of this inequality in this paper. Since the problem is not variational in the natural energy space, classical minimisation or critical-point techniques do not apply directly.  Positive solutions to the resulting semilinear equation are then obtained via a carefully designed Galerkin method applied to this setting. Our main theorem recovers previous known results on $\mathbb{R}^n$ in the complete range $q\in (1,\infty),$ and the previous known result on the  Heisenberg group $\mathbb{H}^n$ for $q=2n+2,$ and is extended here in the full range $q\in (1,\infty)$. Other particular and fundamental cases, including the Engel group, are discussed.
\vspace{0.6cm}

\noindent\textbf{2020 Mathematics Subject Classification:} 35A35, 35J10, 35J62, 35R03.\\

\noindent\textbf{Keywords:} Stratified Lie groups, Schrödinger equation, approximation scheme, Trudinger-Moser inequality, Folland-Stein (Sobolev) spaces, quasilinear elliptic equations.
\end{abstract}

\maketitle

\vspace{0.5cm}

\tableofcontents
\allowdisplaybreaks
\section{Introduction}

\subsubsection{Overview}
Motivated by the recent progress in the study of partial differential equations on stratified Lie groups (see, e.g., the monographs by Ruzhansky and his collaborators \cite{Ruzhansky:Fischer:Book, Ruzhansky:Suragan:Book}), a framework that combines sub-Riemannian geometry with the analysis of hypoelliptic operators, in the present work we investigate the existence of positive weak solutions to the quasilinear equation \eqref{prob:inicial} for the $q$-sub-Laplacian (also called here, the horizontal $q$-Laplacian) with exponential nonlinearity. We consider this to be an interesting nonlinear subelliptic model whose analysis requires delicate spectral inequalities as the version investigated here of the Trudinger Moser inequality for Lorentz spaces and that is adapted to \eqref{prob:inicial}. 

We also recall that the class of stratified Lie groups includes several important geometric structures, such as the Euclidean space $\mathbb{R}^n$, the Heisenberg group $\mathbb{H}^n$, more general Heisenberg-type groups, the Engel group $\mathbb{E}$, and free nilpotent Lie groups arising from free nilpotent Lie algebras; see, e.g., \cite{Le:Donne}  and Corwin and Greenleaf \cite{CorwinGreenleafBook}.

The problem \eqref{prob:inicial} will be associated to the horizontal $q$-Laplacian on a stratified Lie group, with the case $q=2,$ being reduced to the canonical sub-Laplacian associated with a system of left-invariant vector fields satisfying H\"ormander's celebrated rank condition, see  H\"ormander \cite{Hormander}. Sub-Laplacians are the prototype of hypoelliptic operators on stratified groups, and their analysis has been thoroughly developed by Folland \cite{Follan1} and Folland and Stein \cite{FollandStein}, among many others. 

On stratified Lie groups $\mathbb{G}$ of homogeneous dimension $Q$, and in the case $q=Q$, the existence of the fundamental solution for the horizontal $q$-Laplacian was established by Balogh, Manfredi, and Tyson in \cite[Page 41]{BMT2003}. In \cite{ChoudhuriTavaresAlvarez2025}, a critical hypoelliptic problem for the horizontal $q$-Laplacian on stratified Lie groups was investigated; see also Choudhuri and Repovš \cite{ChoudhuriRepovs2023}. For Green's identities, the comparison principle, and the uniqueness of positive solutions for a class of nonlinear $q$-sub-Laplacian equations on stratified Lie groups, we refer the reader to Ruzhansky and Suragan \cite{RuzhanskySuragan2020}. We also refer to  Garofalo and Lanconelli \cite{GarofaloLanconelli1990}  where the authors obtained existence and nonexistence results using Rellich-Pohozaev type inequalities and have investigated uncertainty type principle and unique continuation for the Heisenberg group and $q=2$.

Our motivation to investigate the quasilinear problem \eqref{prob:inicial} on stratified Lie groups here, also comes  from the fundamental {\it Acta Mathematica} work due to Rothschild and Stein \cite{Rothschild:Stein:1976}, who observed  that nilpotent Lie groups play a fundamental role in obtaining sharp subelliptic estimates for differential operators on manifolds. In view of the celebrated Rothschild--Stein lifting theorem, any sum of squares of Hörmander vector fields defined on a manifold can be approximated, in a remarkably precise way, by a left-invariant sub-Laplacian on a suitable stratified Lie group as the considered in this paper. 

The lifting theorem due to Rothschild and Stein \cite{Rothschild:Stein:1976} has made the systematic study of partial differential equations on stratified Lie groups not only natural but genuinely indispensable, and has given rise to a rich and highly active line of research that links Lie theory and analysis of PDEs. Moreover, in recent decades sub-Laplacians on stratified groups have attracted rapidly growing attention for their fundamental approximation properties, in view of the Rothschild and Stein lifting method, as emphasized in the panoramic expositions of Danielli and Nhieu \cite{Danielli:Nhieu:2007}. We also refer the reader to Garofalo and Lanconelli \cite{Garofal1992} and Gromov \cite{Gromov:1996}.

\subsubsection{Quasilinear equations with exponential growth on stratified Lie groups} 
In order to precise the problem considered here, let $\mathbb{G}=(\mathbb{R}^n,\circ)$ be a  \emph{stratified Lie group} of
homogeneous dimension $Q$, let $\Omega\subset \mathbb{G}$ be a  bounded open subset of $\mathbb{G}$ of smooth boundary $\partial\Omega$, let $0<p<q-1,$  and let us consider $\lambda>0$ to be a parameter. We study  the quasilinear problem with homogeneous boundary conditions on $\mathbb{G}$:
\begin{equation}\label{prob:inicial}
\left \{
\begin{aligned}
&-\mathcal{L}_qu =\lambda u^{p}+\text{exp}(\alpha u^{\frac{Q}{Q-1}}),&  && \text{ in } & \Omega, \\
& u>0, & && \text{ in } &\Omega,\\
& u =0 ,& && \text{ on } &\partial \Omega,
\end{aligned}
\right.
\end{equation}
where $\mathcal{L}_q$ is a fixed \emph{horizontal $q$-Laplacian} (or $q$-sub-Laplacian) in $\mathbb{G}$, which is defined as
\[
\mathcal{L}_q (\cdot)
=
\operatorname{div}_\mathbb{G}
\bigl(
|\nabla_\mathbb{G} (\cdot)|^{q-2}\nabla_\mathbb{G} (\cdot)
\bigr),
\qquad
1<q<\infty.
\]
Here, we have denoted by $\nabla_{\mathbb{G}}(\cdot)$ an arbitrary horizontal gradient,  that is,
\begin{equation}\label{H:L:G}
\nabla_{\mathbb{G}}:=(X_1,X_2,...,X_N),
\end{equation}
where $\{X_j\}$, $j=1,\ldots,N$, is a system of H\"ormander vector fields, and then spanning the Lie algebra $\mathfrak{g}\equiv T_0\mathbb{G}$ of $\mathbb{G}.$ We will refer to these a system of horizontal left--invariant vector fields on $\mathbb{G}$. 

To our knowledge, the study of quasilinear problems with exponential growth terms has not been considered in the general setting of  H\"ormander $q$--sub--Laplacians, with the case of the Heisenberg group investigated by the second author and Agila in \cite{Tordecilla:Aguila:2025} and in the Euclidean space \cite{Araujo, Tordecilla}. Our conjecture is that the methods in \cite{Tordecilla:Aguila:2025} can be extrapolated to the more general setting of stratified Lie groups. 

However, it is important to emphasise that the framework of stratified Lie groups cannot be straightforwardly adapted from the approaches and techniques developed for the Heisenberg group. This is due to several additional considerations related to Sobolev--type inequalities, Trudinger--type inequalities, and other analytical tools, which in the setting of  stratified Lie groups must be carefully interfaced with the analysis of the $q$--sub--Laplacian.

\subsubsection{Some functional inequalities}
 It has been known since the work of Varopoulos \cite{Varopoulos1, Varopoulos2} and Saloff--Coste (see also Folland \cite{Follan1}), that the following version of the Sobolev inequality
holds on $\mathbb{G}$:
\begin{equation}\label{Sov}
\left(
\int_{\mathbb{G}} |f(x)|^s\,dx
\right)^{1/s}
\leq
C_{p,s}
\left(
\int_{\mathbb{G}} |\nabla_{\mathbb{G}} f(x)|^p\,dx
\right)^{1/p},
\end{equation}
provided that $1\leq p<Q$, and $\frac{1}{p}-\frac{1}{s}=\frac{1}{Q}$, where $|\nabla_{\mathbb{G}} f|$ stands for the norm of the
\emph{horizontal gradient} in \eqref{H:L:G} of a function
$f\in C_0^\infty(\mathbb{G})$. By completion of
$C_0^\infty(\mathbb{G})$ under the Sobolev norm $\|f\|_p+\|\nabla_{\mathbb{G}} f\|_p$, the above inequality holds for functions in the horizontal Sobolev space $W_0^{1,p}(\mathbb{G})$.

Here we observe, as pointed out in \cite{Ruzhansky:Suragan:Book}, that the Sobolev spaces $W^{r,p}(\mathbb{G})$ (also known as Folland--Stein spaces) on stratified Lie groups were first introduced by Folland \cite{Follan1}. Several further properties were then developed in the book by Folland and Stein \cite{FollandStein}. For a comprehensive treatment of the subject, we refer to Ruzhansky and Suragan \cite{Ruzhansky:Suragan:Book}.

Observe that in the case $p=Q,$ the Sobolev inequality~\eqref{Sov} turns into the Trudinger
inequality stated as follows. There exist constants $\alpha_Q>0$ y $
c_0>0$ such that for any domain $
\Omega\subset \mathbb{G}$, $|\Omega|<\infty$, and
$f\in W_0^{1,Q}(\Omega)$, the following inequality holds:
\begin{equation}\label{TM}    
\frac{1}{|\Omega|}
\int_{\Omega}
\exp\left(
\alpha_Q
\frac{|f(x)|^{Q'}}
{\|\nabla_{\mathbb{G}} f\|_Q^{Q'}}
\right)\,dx
\leq c_0,
\end{equation}
where $Q'=\frac{Q}{Q-1}$ is the dual exponent of $Q$ (see \cite[p. 36]{balogh2003fundamental}). 

\subsubsection{Main result}
The following is our main theorem. 

\begin{thm}\label{T.Gb} Let $\mathbb{G}$ be a stratified Lie group and let $1<q<\infty.$
Let us consider $\Omega\subset\mathbb{G}$ be a  bounded open domain with smooth boundary $\partial\Omega$, and fix  $p$ such that $0<p<q-1$. Then there exist $\lambda^*:=\lambda^*(\Omega)>0$, and $\alpha^*:=\alpha^*(\Omega)>0$, $($depending only on $p$, $Q$, $q$, and $\Omega$$)$ such that, for every
$0<\lambda<\lambda^*$, and all $0<\alpha<\alpha^*$, the subelliptic boundary-value problem 
\begin{equation}
\left \{
\begin{aligned}
&-\mathcal{L}_qu =\lambda u^{p}+\text{exp}(\alpha u^{\frac{Q}{Q-1}}),&  && \text{ in } & \Omega, \\
& u>0, & && \text{ in } &\Omega,\\
& u =0 ,& && \text{ on } &\partial \Omega,
\end{aligned}
\right.
\end{equation}
admits at least one positive weak solution $u\in W_0^{1,q}(\Omega)$.
\end{thm}
Now, we briefly discuss our result. 
\begin{remark}
    In order to briefly discuss our main result we recall that in the case of the Heisenberg group, and for $q=Q,$ Theorem \ref{T.Gb} has been obtained in \cite{Tordecilla:Aguila:2025}. For other values of $q$ our result is already new on the Heisenberg group.  Our theorem also recovers the corresponding result for the Euclidean space. We refer the reader to the introduction in \cite{Tordecilla:Aguila:2025} for a discussion about these cases. Theorem \ref{T.Gb} consider the case where $Q\geq 2.$ Nevertheless, the cases $Q=1,2$ can be
avoided since they do not produce a purely stratified Lie group. 
\end{remark}
 \begin{figure}[ht]
    \centering
    \includegraphics[width=0.65\textwidth]{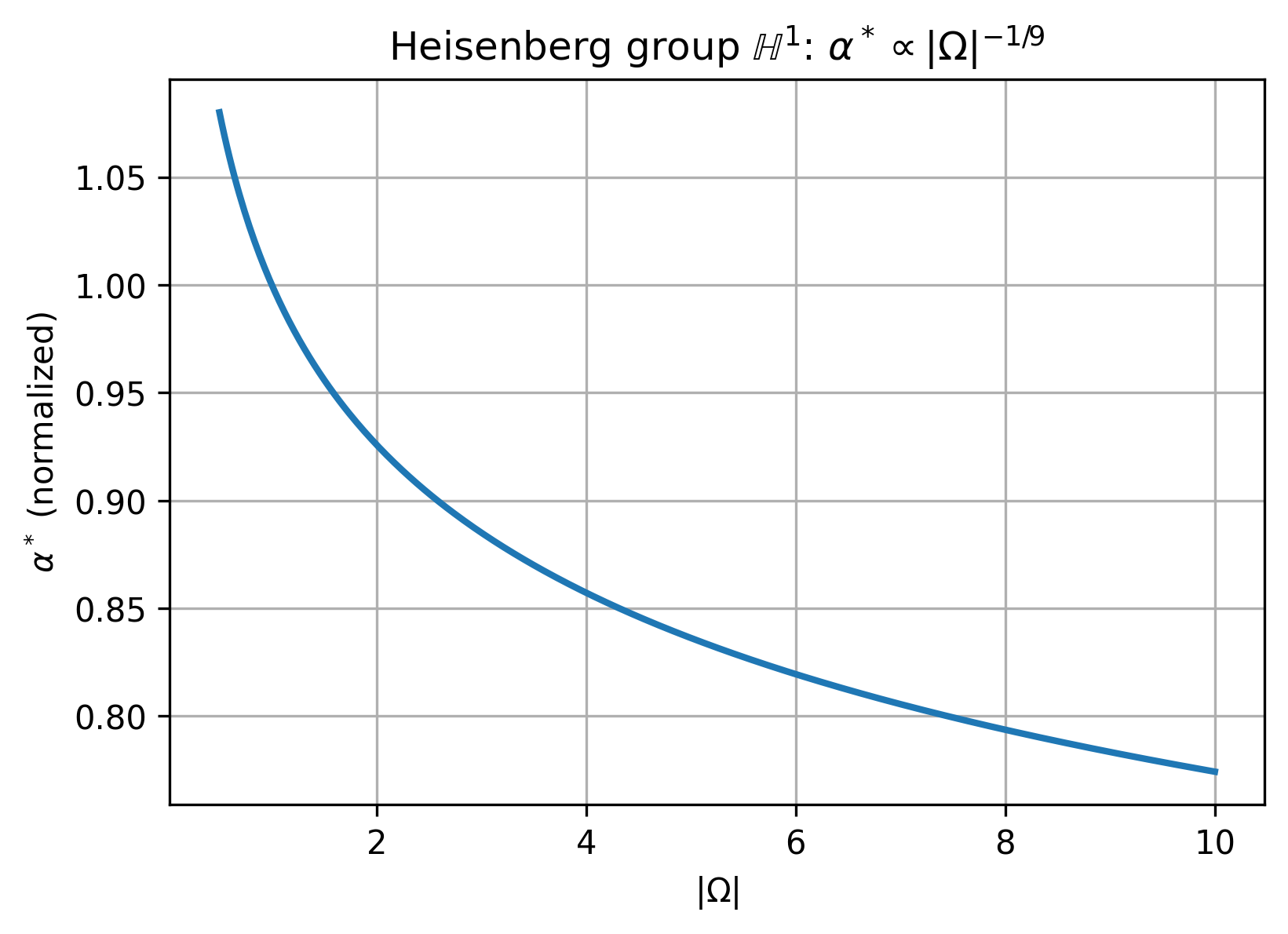}
    \caption{Dependence of the normalised threshold $\alpha^*$ on the measure of the domain $|\Omega|$ for the Heisenberg group ($Q=4$). The relationship follows the power law $\alpha^*\asymp |\Omega|^{-1/9}$, illustrating the gradual decrease of the threshold as the domain size increases.}
    \label{fig:alpha-heisenberg}
\end{figure} 
\begin{remark} We remark that the parameter $\alpha^*$ in Theorem \ref{T.Gb} encodes the size of the domain $\Omega.$ In the particular case of the Heisenberg group \(\mathbb{H}^1\) whose homogeneous
dimension is \(Q=4\),  see Figure \ref{fig:alpha-heisenberg}, the threshold satisfies
\[
\alpha^* \asymp |\Omega|^{-1/9}.
\] This analysis comes from the method of our proof and the definition of $\alpha^*$ in \eqref{alpa} leaving other constants ``normalised". We call to this new parameter, the normalised parameter $\alpha^*.$  
Hence, the dependence of \(\alpha^*\) on the measure of the domain is relatively
mild, showing that even significant changes in \(|\Omega|\) produce only moderate
variations in the admissible threshold. A similar analysis can be made for the parameter $\lambda^*$ in Theorem  \ref{T.Gb}.
\end{remark}
\begin{figure}[ht]
    \centering
    \includegraphics[width=0.65\textwidth]{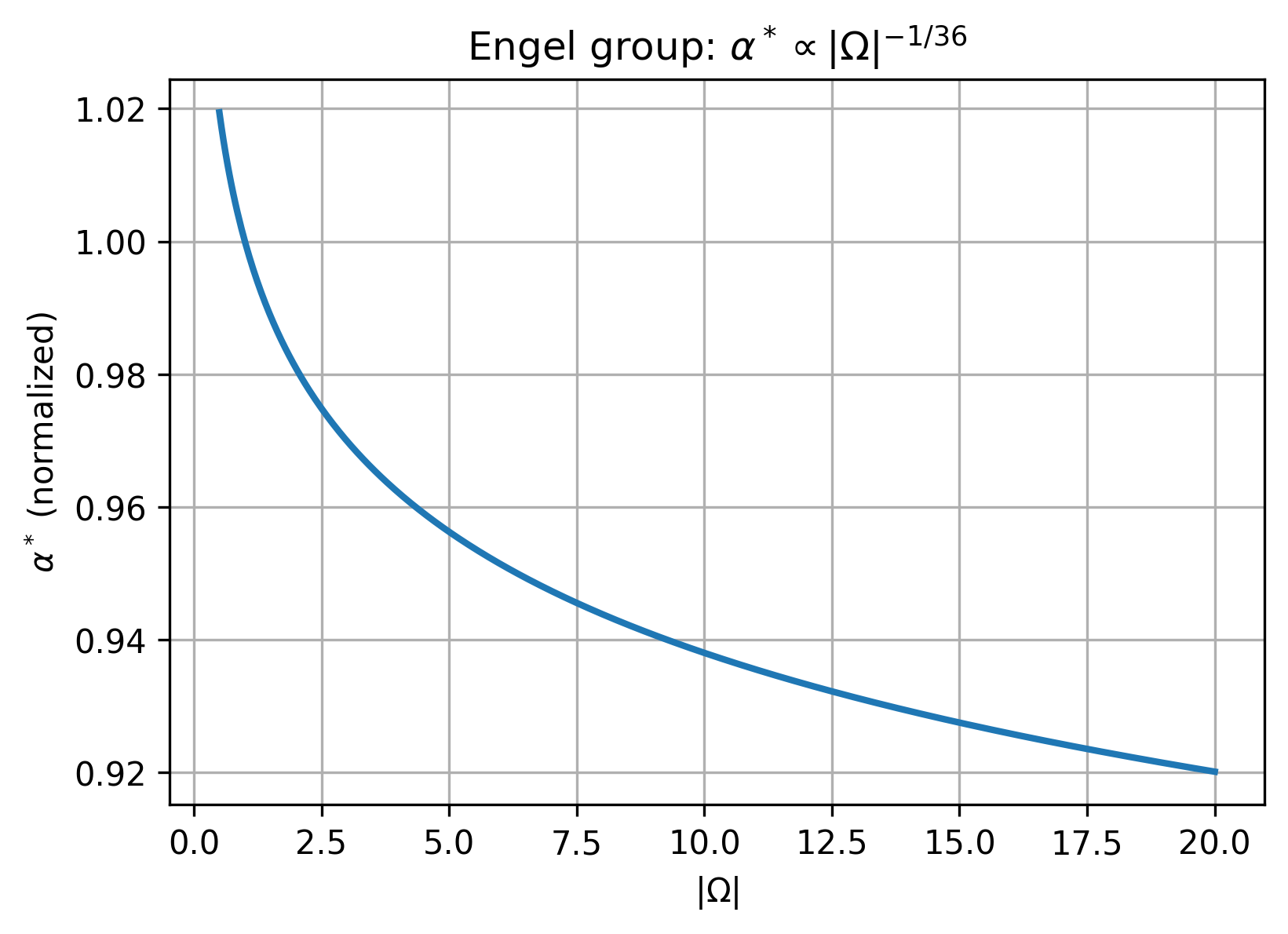}
    \caption{Normalised threshold $\alpha^*$ as a function of the domain measure $|\Omega|$ for the Engel group ($Q=7$). The dependence follows the estimate $\alpha^*\asymp |\Omega|^{-1/36}$.}
    \label{fig:alpha-engel}
\end{figure}
\begin{remark}\label{css}
As in the case of the Euclidean space, note that $\alpha <\alpha_Q$, $\alpha = \alpha_Q$ and $\alpha > \alpha_Q$ represent the subcritical, critical, and supercritical  growth of $\text{exp}(\alpha |u|^{\frac{Q}{Q-1}})$ as $|u|\to +\infty$, respectively. In our further analysis, an extension of the Trudinger Moser inequality \eqref{TM}  to Lorentz spaces is required, see Lemma \ref{lem4}.
\end{remark}
\begin{remark}
    Note that a similar analysis for the parameters $\alpha^*$ and $\lambda^*$ in Theorem  \ref{T.Gb}  can be also done, as in Figure \ref{fig:alpha-heisenberg}, for example:
    
    $ \bullet$ in the case of the Engel group $\mathbb{E}$ of homogeneous dimension then Q=7. 
Therefore, from \eqref{alpa}, one has that (see Figure \ref{fig:alpha-engel}),
\[
\alpha^*
\asymp
|\Omega|^{-1/36}.
\]
Hence, the threshold $\alpha^*$ decreases monotonically with the measure of the
domain, although, clearly, at a significantly slower rate than in the case of the Heisenberg group. 

    $\bullet$ For more general Heisenberg-type groups, as the quaternionic Heisenberg group $\mathbb{H}^n_q$ of homogeneous dimension $Q=4n+3.$ Note that the case of $\mathbb{H}^2_q$ and of $\mathbb{H}^3_q$ and  gives $Q=11$ and $Q=15,$ respectively.
    
    $\bullet$ For general free nilpotent Lie groups arising from free nilpotent Lie algebras; see, e.g., \cite{Le:Donne}. In the latter case, a free nilpotent Lie group $\mathbb{G}_{s,m}$ with Lie algebra with $m$ generators and of step $s,$ has dimension given by $Q=\sum_{j=1}^s\sum_{j|d}\mu(d)j^{m/d},$ where $\mu(\cdot)$ denotes the M\"obius function. 
\end{remark}
\begin{remark}
     In all these cases, on an arbitrary stratified Lie group $\mathbb{G}$ of homogeneous dimension $Q,$ the behavior expected for the geometric constant $\alpha^*$ is $\alpha^*\asymp |\Omega|^{-\frac{1}{(Q-1)^2}},$ and for the spectral parameter $\lambda^*$ is $\lambda^*\asymp |\Omega|^{-1}.$ This discussion in any case is heuristic since in the definition of these parameters other constants appear but by fixing them the numerical analysis of these constants can be improved.  In particular, when two domains $\Omega$ and $\Omega'$ have the same volume 
$|\Omega|=|\Omega'|$, one expects their associated geometric and spectral 
parameters $\alpha^*=\alpha^*(\Omega),$  $\alpha^{\bullet}=\alpha^*(\Omega'),$ and $\lambda^*=\lambda^*(\Omega),$ $\lambda^{\bullet}=\lambda^*(\Omega'),$  to be comparable, respectively.
\end{remark}
    For further developments in the analysis of hypoelliptic problems on graded Lie groups, which constitute a broader class than stratified Lie groups, we refer the reader to \cite{Cardona:Delgado:Ruzhansky:2021,Cardona:Ruzhansky:2024,Cardona:Kumar:Ruzhansky:2025,Cardona:Delgado:Ruzhansky:2022}. These works develop a comprehensive framework for harmonic analysis, pseudo--differential operators, spectral multipliers, function spaces, and evolution equations associated with hypoelliptic operators on graded Lie groups.

\subsubsection{Organisation of the paper}
The preliminaries required for this manuscript are presented in the next section. More precisely, the paper is organised as follows.

$\bullet$ In Section \ref{sec1}, we introduce the notions related to stratified Lie groups that are used throughout the paper. In particular, we present the function spaces adapted to the Hörmander sub-Laplacian on these groups.

$\bullet$ Section \ref{Sec3CT} is divided in two parts. In Subsection \ref{Sub:31}, we prove a version of the Trudinger--Moser inequality tailored to our subsequent analysis, namely, a version adapted to Lorentz spaces. Finally, in Subsection \ref{sec3}, we prove our main result, Theorem \ref{T.Gb}.

\section{Preliminaries and auxiliary results}\label{sec1}

\subsection{Stratified Lie groups}

In this section, we present the notion of {\it stratified Lie group} and we present some auxiliary results that will be used throughout the paper. 
\subsubsection{Preliminaries} We start by presenting the notion of homogeneous Lie group and of its homogeneous dimension. 
\begin{defn}
A Lie group $\mathbb{G}$ (on $\mathbb{R}^n$) is said to be
\emph{homogeneous} if, for each $\lambda>0$, there exists an automorphism $D_\lambda:\mathbb{G}\to\mathbb{G}$, defined by $$D_\lambda(x)=
(\lambda^{r_1}x_1,\lambda^{r_2}x_2,\ldots,\lambda^{r_N}x_N),$$ for $r_i>0$, $\forall\, i=1,2,\ldots,N$. The map $D_\lambda$ is called a
\emph{dilation} on $\mathbb{G}$.
\end{defn}

For simplicity, we sometimes prefer to use the notation $\lambda x$ to denote the dilation $D_\lambda x$. Note that, if $\lambda x$ is a
dilation on $\mathbb{G}$ then $\lambda^r x$ is also a dilation on the group. 

The number $$Q=r_1+r_2+\cdots+r_N,$$ is called the \emph{homogeneous dimension} of the homogeneous Lie group
$\mathbb{G}$ and the natural number $n$ represents the topological
dimension of $\mathbb{G}$. 

The Haar measure on $\mathbb{G}$ is denoted by $dx$ and it is nothing
but the usual  measure induced by the Lebesgue measure on $\mathbb{R}^N$.

\begin{defn}[Stratified Lie groups]
A Lie group $\mathbb{G}=(\mathbb{R}^n,\circ)$ is called a \emph{stratified group} (or a \emph{homogeneous Carnot group}) if it satisfies the following conditions:

\begin{enumerate}
\item[(a)] For some natural numbers $N=N_1, N_2,\cdots,N_r,$ the decomposition
\[
\mathbb{R}^n=\mathbb{R}^N\times\cdots\times\mathbb{R}^{N_r},
\]
is valid, and then $n=N+N_2+\cdots, N_r$, and for every $\lambda>0$ the dilation
\[
\delta_\lambda:\mathbb{R}^n\to\mathbb{R}^n
\]
given by
\[
\delta_\lambda(x)\equiv
\delta_\lambda(x',x^{(2)},\ldots,x^{(r)})
:=
(\lambda x',\lambda^2x^{(2)},\ldots,\lambda^r x^{(r)}),
\]
is an automorphism on the group $\mathbb{G}$. Here, we denote
\[
x'=x^{(1)}\in\mathbb{R}^N
\quad \text{and} \quad
x^{(k)}\in\mathbb{R}^{N_k},
\qquad k=2,\ldots,r.
\]

\item[(b)] Let $N$ be as in (a) and let $X_1,\ldots,X_N$, be left invariant vector fields on $\mathbb{G}$ such that
\[
X_k(0)=\left.\frac{\partial}{\partial x_k}\right|_{0},
\qquad k=1,\ldots,N,
\]
and satisfying the rank condition
\[
\operatorname{rank}\big(\mathrm{Lie}\{X_1,\ldots,X_N\}\big)=n,
\]
this means that the iterated commutators of
$X_1,\ldots,X_N,$ span the Lie algebra of $\mathbb{G}$.
\end{enumerate}
\end{defn}

The number $r$ is called the \emph{step} of $\mathbb{G}$, and the left invariant vector fields
$X_1,\ldots,X_N$, are called  Hörmander vector fields generating the Lie algebra of $\mathbb{G}$. Note that the homogeneous dimension
of a stratified Lie group $\mathbb{G}$ is also given by $${Q=\sum_{k=1}^{r} kN_k,\quad N_1=N}.$$ 
The second-order differential operator
\begin{equation}\label{defl}
\mathcal{L}=\sum_{k=1}^{N} X_k^2,
\end{equation}
is called the Hörmander sub-Laplacian on $\mathbb{G}$. The sub-Laplacian
$\mathcal{L}$ is a left invariant homogeneous hypoelliptic differential operator
and it is elliptic if and only if the step of $\mathbb{G}$ is equal to $1$, or in other words if $\mathbb{G}=\mathbb{R}^N$ as a group.

The \emph{hypoellipticity} of $\mathcal{L}$ means that for a distribution
$f\in \mathcal{D}'(\Omega)$ in any open set $\Omega$, if
$\mathcal{L}f\in C^\infty(\Omega)$ then $f\in C^\infty(\Omega)$.
It is a special case of Hörmander’s sum of squares theorem \cite{Hormander}.

The left invariant vector field $X_k$ has an explicit form given in Proposition~1.2.19, namely,
\begin{equation}\label{campo}
X_k
=
\frac{\partial}{\partial x_k'}
+
\sum_{l=2}^{r}\sum_{m=1}^{N_l}
a_{k,m}^{(l)}
\bigl(x',\ldots,x^{(l-1)}\bigr)
\frac{\partial}{\partial x_m^{(l)}},
\end{equation}
where $a_{k,m}^{(l)}$ is a homogeneous (with respect to $\delta_\lambda$)
polynomial function of degree $l-1$. 

We will also use the following notation
for the \emph{horizontal gradient}
\[
\nabla_\mathbb{G} = (X_1,\ldots,X_N),
\]
for the \emph{horizontal divergence}
\[
\operatorname{div}_\mathbb{G} v := \nabla_\mathbb{G} \cdot v\equiv \sum_{j=1}^NX_jv_j,\,\,v:=(v_j)_{1\leq j\leq N}\in C^\infty(\mathbb{G},\mathbb{C}^N),
\]
and for the \emph{horizontal $q$-Laplacian} (or $q$-sub-Laplacian)
\[
\mathcal{L}_q (\cdot)
=
\operatorname{div}_\mathbb{G}
\bigl(
|\nabla_\mathbb{G} (\cdot)|^{q-2}\nabla_\mathbb{G} (\cdot)
\bigr),
\qquad
1<q<\infty.
\]

\begin{remark}
Denoting the Euclidean distance by
\[
|x'|
=
\sqrt{(x_1')^2+\cdots+(x_N')^2}.
\]
for the Euclidean norm on $\mathbb{R}^N$, the representation~\eqref{campo}
for derivatives leads to the identities
\[
|\nabla_\mathbb{G} |x'|^\gamma|
=
\gamma |x'|^{\gamma-1},
\]
and
\[
\operatorname{div}_H
\left(
\frac{x'}{|x'|^\gamma}
\right)
=
\frac{
\sum_{j=1}^{N}|x'|^\gamma X_jx_j'
-
\sum_{j=1}^{N}x_j'\gamma |x'|^{\gamma-1}X_j|x'|
}{
|x'|^{2\gamma}
}
=
\frac{N-\gamma}{|x'|^\gamma},
\]
for all $\gamma\in\mathbb{R}$, $|x'|\neq 0$.
\end{remark}

An absolutely continuous curve $\gamma : [0,1] \to \mathbb{R}$, is said to be admissible, if there exist functions $c_i : [0,1] \to \mathbb{R}$, for $i = 1,2,\ldots,n$, such that
\[
\dot{\gamma}(t)=\sum_{i=1}^{n} c_i(t)X_i(\gamma(t)),
\quad \text{and} \quad
\sum_{i=1}^{n} c_i(t)^2 \leq 1.
\]

Observe that the functions $c_i$ may not be unique as the vector fields $X_i$ may not be linearly independent. For any $x,y \in \mathbb{G}$ the Carnot--Carathéodory distance is defined as
\begin{align*}
\rho_{cc}(x,y)
&=
\inf \left\{
l>0 :
\exists
\gamma : [0,l] \to \mathbb{G} \text{ admissible,}\,\,
\gamma(0)=x,
\gamma(l)=y
\right\}.
\end{align*}

We define $\rho_{cc}(x,y)=0$, if such curve does not exists. Note that $\rho_{cc}$ is not a metric in general but the Hörmander condition for the vector fields
$X_1,X_2,\ldots,X_N,$
ensures that $\rho_{cc}$ is a metric (see \cite{Hormander}). The space $(\mathbb{G},\rho_{cc})$ is known as a Carnot--Carathéodory space.

It was shown by Folland \cite{Follan1} that the sub-Laplacian
$\mathcal{L}$ in~\eqref{defl} on a general stratified group $\mathbb{G}$
has a unique \emph{fundamental solution} $\Gamma$, that is,
\[
\mathcal{L}\Gamma=\delta,
\]
where $\delta$ is the delta-distribution at the unit element of
$\mathbb{G},$ and with the function $\Gamma$ being homogeneous
of degree $2-Q$.

The function
\[
d(x):=
\begin{cases}
\Gamma(x)^{\frac{1}{2-Q}},
& \text{for } x\neq 0,
\\
0,
& \text{for } x=0,
\end{cases}
\]
is called the $\mathcal{L}$-gauge on $\mathbb{G}$. It is a
\emph{homogeneous quasi-norm} on $\mathbb{G}$, that is, it is a
continuous function
\[
d:\mathbb{G}\to [0,\infty),
\]
smooth away from the origin, which satisfies the conditions
\begin{enumerate}
    \item[i)] $d(x)=0$ if only if $x=0$;
    \item[ii)] $d(x^{-1})=d(x)$ for all $x\in \mathbb{G}$;
    \item[iii)]  $d(\beta x)=\beta d(x)$ for all $x\in \mathbb{G}$ and $\beta>0$. 
\end{enumerate}

We refer to the original paper \cite{Follan1} by Folland as well as to
a recent presentation in \cite[Section 3.2.7]{Ruzhansky:Fischer:Book} for further details
and properties of these fundamental solutions. Moreover, for the boundedness properties for the parametrix of the sub-Laplacian and other pseudo-differential operators on stratified Lie groups and other Lie groups, we refer the reader to \cite{Cardona:Delgado:Ruzhansky:2021}.

Finally, we record some facts about the action of dilations on the measure of measurable sets.
Let $\Omega$ be a Haar measurable subset of $\mathbb{G}$. Then
\[
\nu(D_\beta(\Omega))=\beta^Q \nu(\Omega)
\]
where $\nu(\Omega)$ is the Haar measure of $\Omega$. The quasi--ball of radius $r$ centered at $x \in \mathbb{G}$ with respect to the quasi--norm $d(\cdot)$ is defined as
\begin{equation*}
B(x,r)=\left\{ y\in \mathbb{G} : \left| y^{-1}\circ x \right| < r \right\}.
\end{equation*}

Observe that $B(x,r)$ can be obtained by the left--translation by $x$ of the ball $B(0,r)$. Furthermore, $B(0,r)$ is the image under the dilation $D_r$ of $B(0,1)$. Thus, we have
\[
\nu(B(x,r)) = r^Q
\]
for all $x \in \mathbb{G}$. For a measurable set $A\subset \mathbb{G},$ we note by $\nu(A)=|A|$ its Haar measure.

\subsection{Extended sub--Laplacians}

In general, most of the results described in this paper in the setting
of stratified groups can be extended to any second-order hypoelliptic
differential operators which are ``equivalent'' to the sub-Laplacian
$\mathcal{L}$. Let us very briefly discuss this matter in the spirit
of \cite{Bonfiglioli} and the reason making natural this extension. Indeed, let
\[
A=(a_{k,j})_{1\le k,j\le N_1}
\]
be a positive-definite symmetric matrix. Consider the following
second-order hypoelliptic differential operator based on the matrix
$A$, and the vector fields $\{X_1,\ldots,X_{N_1}\}$ from the first
stratum, given by
\[
\mathcal{L}_A
=
\sum_{k,j=1}^{N_1}
a_{k,j}X_kX_j.
\]
For instance, in the Euclidean case, that is, for
$\mathbb{G}=(\mathbb{R}^N,+)$ and $N_1=N$, the constant coefficients
second-order elliptic operator
\[
\Delta_A
=
\sum_{k,j=1}^{N}
a_{k,j}
\frac{\partial^2}{\partial x_k\partial x_j}
\]
is transformed into the Laplacian
\[
\Delta
=
\sum_{k=1}^{N}
\frac{\partial^2}{\partial x_k^2}
\]
under a linear change of coordinates in $\mathbb{R}^N$. Thus, the
operator $\Delta_A$ is ``equivalent'' to the operator $\Delta$ by a
linear change of the coordinate system.

In general, to apply the above argument to transform $\mathcal{L}_A$
to the sub-Laplacian $\mathcal{L}$ it is not enough to change the basis
by a linear transformation. However, it is enough in the setting of
free stratified groups. We say that a stratified group $\mathbb{G}$
is a \emph{free stratified group} if its Lie algebra is (isomorphic to)
a free Lie algebra. For instance, the Heisenberg group $\mathbb{H}^1$
is a free stratified group. In this case we have the following result.

\begin{thm}
Let $\mathbb{G}$ be a free stratified group and let $A$ be a given
positive-definite symmetric matrix. Let
\[
X=\{X_1,\ldots,X_{N_1}\}
\]
be left invariant vector fields in the first stratum of the Lie algebra
of $\mathbb{G}$. Let
\[
Y_k
:=
\sum_{j=1}^{N_1}
\left(A^{\frac12}\right)_{k,j}X_j,
\qquad
k=1,\ldots,N_1.
\]

Consider the associated second-order differential operator
\[
\mathcal{L}_A
=
\sum_{k=1}^{N_1}Y_k^2
=
\sum_{k,j=1}^{N_1}
a_{k,j}X_kX_j.
\]

Then there exists a Lie group automorphism $T_A$ of $\mathbb{G}$ such that
\[
Y_k(u\circ T_A)
=
(X_ku)\circ T_A,
\qquad
k=1,\ldots,N_1,
\]
and
\[
\mathcal{L}_A(u\circ T_A)
=
(\mathcal{L}u)\circ T_A,
\]
for every smooth function
\[
u:\mathbb{G}\to\mathbb{R}.
\]
Moreover, $T_A$ has polynomial component functions and commutes with the dilations of $\mathbb{G}$.
\end{thm}

\subsection{Classical Sobolev spaces on stratified Lie groups}

For $\Omega\subset \mathbb{G}$ an open subset, we consider the subelliptic Sobolev space
\[
W^{1,p}(\Omega):= W^{1,p}_{\mathcal{L}}(\Omega)
=
\left\{
u : \Omega \to \mathbb{R}
\; ; \;
u,\ |\nabla_{\mathbb{G}}u| \in L^p(\Omega)
\right\}, \quad 0<p<\infty.
\] This is a natural Sobolev space associated to the sub-Laplacian $\mathcal{L}=\sum_{k=1}^{N} X_k^2.$
Moreover, let us consider the following functional
\[
J_p(u)
:=
\left(
\int_{\Omega} |\nabla_{\mathbb{G}}u|^p \, dx
\right)^{\frac{1}{p}}.
\]
Then we define the space $W_0^{1,p}(\Omega)$ to be the completion of $C_0^1(\Omega)$ in the norm generated by $J_p$ (see, e.g. Capogna, Danielli, and Garofalo \cite{Capogna}). Note that for $1\leq p<\infty,$ the Sobolev space $W^{1,p}(\Omega)$ is a Banach space endowed with the norm
\begin{equation}
    \Vert u \Vert_{ W^{1,p}}:=  \Vert u \Vert_{ W^{1,p}_\mathcal{L}}= J_p(u).
\end{equation}Moreover, we have the equivalent of norms
\begin{equation}
    \Vert u \Vert_{ W^{1,p}}\asymp  \Vert u \Vert_{ W^{1,p}}':=\sum_{j=1}^N\|X_j u\|_{L^p(\Omega)}.
\end{equation} Note that the Sobolev space $W^{1,p}_{\mathcal{L}}(\mathbb{G})$ does not necessarily agrees with the standard Sobolev spaces $H^{1,p}(\mathbb{R}^n)$ defined by standard partial derivatives on $\mathbb{R}^n$ even considering the identification $\mathbb{G}\cong \mathbb{R}^n,$ although embedding between them are possible, we refer e.g. to the paper of the second author and Ruzhansky \cite{CRPA} for this discussion even for graded Lie groups.

\section{Proof of the main theorem}\label{Sec3CT}

\subsection{Lorentz spaces and the subelliptic Trudinger–Moser inequality}\label{Sub:31} In our further analysis we will require a suitable subelliptic (horizontal version of the) Trudinger–Moser inequality on Lorentz spaces  $L(p,q):=L^{p,q}(\mathbb{G}) $ on strafied Lie groups $\mathbb{G}$. Under the identification $\mathbb{G}\cong \mathbb{R}^n,$ where $n$ is the topological dimension of $\mathbb{G},$ these spaces can be defined in a similar way to the one settled for the Euclidean space. We recall this construction in the next subsection.  
\subsubsection{Auxiliary results}
We recall the definition of the Lorentz space on the stratified group $\mathbb{G}$, see e.g. \cite{Follan1}. For a measurable function
$\phi : \Omega \to \mathbb{R}$, we denote by
\[
\mu_{\phi}(t)=\big|\{x\in\Omega:|\phi(x)|>t\}\big|,
\qquad \text{for } t\geq 0,
\]
its distribution function, where $|\Omega|$ denotes the Haar measure of $\Omega\subset \mathbb{G}$. The decreasing rearrangement $\phi^{*}$ of
$\phi$ is defined by
\[
\phi^{*}(\beta)=\sup\{t>0:\mu_{\phi}(t)>\beta\},
\qquad 0\leq \beta\leq |\Omega|.
\]
The Lorentz space $L(\tau,q)$ is defined by

\begin{align*}
L(\tau,q)
&=
\{
\phi:\Omega\to\mathbb{R}
\ \text{measurable}:
\ 1<\tau<\infty,\; 1\leq q<\infty,\\
&
\left(
\int_{0}^{|\Omega|}
\frac{1}{t}
\big[\phi^{*}(t)t^{1/\tau}\big]^q
\,dt
\right)^{1/q}
<\infty\},
\end{align*}
endowed with norm
\[
\|\phi\|_{L(\tau,q)}
=
\left(
\int_{0}^{|\Omega|}
\frac{1}{t}
\big[\phi^{*}(t)t^{1/\tau}\big]^q
\,dt
\right)^{1/q}.
\]
We recall in the sequel a few properties of the Lorentz spaces according
to \cite{Adan, Tartar} on the Euclidean space. We have the identity $L(\tau,\tau)=L^{\tau}$ for all $1<\tau<\infty$, and $\|\phi\|_{L(\tau,\tau)}=\|\phi\|_{L^{\tau}}$. Note that, when $\tau=\infty$ and $1<s<\infty$, the Lorentz space $L(\infty,s)$ is given by $$L(\infty,s):=\{\phi:\Omega\to\mathbb{R}
\ \text{measurable}:\|\phi\|_{L(\infty,s)}<\infty\},$$ where 
\[
\|\phi\|_{L(\infty,s)}:=\left(\int_0^{|\Omega|}
\left(
\phi^{**}(t)-\phi^*(t)
\right)^s
\frac{dt}{t}\right)^{\frac{1}{s}}.
\]
The following continuous embedding 
\begin{equation}\label{eq18}
L^{\tau}\hookrightarrow L(Q,1)\hookrightarrow L(Q,s)\hookrightarrow
L(Q,Q)=L^{Q}\hookrightarrow L(Q,\tau)\hookrightarrow L^{s},
\end{equation} holds for all $1<s<Q<\tau<\infty.$ 
The Hölder inequality takes the form:
\[
\left|\int_{\Omega}fg\,dx\right|
\leq
\|f\|_{L(\tau,s)}
\|g\|_{L(\tau',s')},
\]
where $\tau'=\frac{\tau}{\tau-1}$ and $s'=\frac{s}{s-1}.$ Also, by \eqref{eq18}, for all $1<s<Q<\tau<\infty$, we obtain that
\begin{equation}
\|u\|_{L(Q,1)}
\leq
D_{Q,\tau}\,
|\Omega|^{\frac{1}{Q}-\frac{1}{\tau}}
\|u\|_{L(\tau,\tau)},
\,\,
\text{where}
\,\,
D_{Q,\tau}
=
\left[
\frac{1}
{\tau'\left(\frac{1}{Q}-1\right)+1}
\right]^{\frac{1}{\tau'}}.
\tag{19}
\end{equation}

Indeed, let $\tau'=\frac{\tau}{\tau-1}$, then
\begin{align*}
\|u\|_{L(Q,1)}
&=
\int_{0}^{|\Omega|}
u^{*}(t)t^{1/Q}\frac{dt}{t}
\\[1ex]
&=
\int_{0}^{|\Omega|}
u^{*}(t)t^{1/\tau}
\,t^{\,1/Q-1/\tau-1}\,dt
\\[1ex]
&\leq
\left(
\int_{0}^{|\Omega|}
\bigl[u^{*}(t)t^{1/\tau}\bigr]^\tau
\frac{dt}{t}
\right)^{1/\tau}
\left(
\int_{0}^{|\Omega|}
t^{\,\tau'(1/Q-1)}
\,dt
\right)^{1/\tau'}
\\[1ex]
&\leq
\|u\|_{L(\tau,\tau)}
\left[
\frac{
|\Omega|^{\,\tau'(1/Q-1)+1}
}{
\tau'(1/Q-1)+1
}
\right]^{1/\tau'}
\\[1ex]
&=
D_{Q,\tau}\,
\|u\|_{L(\tau,\tau)}
\,|\Omega|^{\frac1Q-\frac1\tau}.
\end{align*}
Since the properties of Lorentz spaces (see, e.g. \cite{Oneil}) provide the estimate
\begin{equation}\label{eq20}
\|u\|_{L(Q,s)}
\leq
\left(\frac{1}{Q}\right)^{1-\frac{1}{s}}
\|u\|_{L(Q,1)},
\end{equation}
by \eqref{eq18} and \eqref{eq20}, we get
\begin{equation}\label{eqqn}
\forall u\in L^{\tau}(\Omega), \quad \|u\|_{L(Q,s)}
\leq
C_{Q,\tau,s}
\,|\Omega|^{\frac{1}{Q}-\frac{1}{\tau}}
\|u\|_{L^{\tau}(\Omega)},
\end{equation}
where $C_{Q,\tau,s}
:=
\left(\frac{1}{Q}\right)^{1-\frac{1}{s}}
D_{Q,\tau}$.

\begin{remark}
On a stratified Lie group $\mathbb{G}$ with homogeneous dimension $Q$, the definition of Lorentz spaces is completely analogous to the Euclidean case, replacing the Lebesgue measure on $\mathbb{R}^N$  by the Haar measure $dx$ of $\mathbb{G}$.
\end{remark}

The following lemma will be useful for our further analysis.

\begin{lem}\label{lemc}
Suppose that $|\nabla_{\mathbb{G}} u| \in L(Q,s)$. Then, there exists a positive constant $C>0$, such that
\begin{equation}\label{Lpestimate}
\|u\|_{L^\tau}
\le
C\,\tau^{1/s'}
\|\nabla_{\mathbb G}u\|_{L(Q,s)},
\qquad
\tau\ge1.
\end{equation}
\end{lem}

\begin{proof}
 The proof is divided into several steps. We proceed as follows.
 
\noindent
{\bf Step 1. Representation formula.} Since $u\in C_0^\infty(\Omega)$, the representation formula, described in Folland and Stein \cite[Chapter 1]{FollandStein}, implies that
\[
u(x)
=
\int_{\mathbb G}
\nabla_{\mathbb G}\Gamma
\left(y^{-1}\circ x\right)
\cdot
\nabla_{\mathbb G}u(y)
\,dy,
\]
where $\Gamma(x)=c_Qd(x)^{2-Q}$ denotes the fundamental solution of the sub--Laplacian on $\mathbb G$  such that $-\Delta_{\mathbb{G}}\Gamma=\delta_0(x)$ and furthermore $|\Delta_{\mathbb{G}}\Gamma (x)|\leq Cd(x)^{1-Q}$. Here $d$ denotes the homogeneous distance in the Carnot group associated to the Hörmander system of vector fields $\{X_1,\ldots,X_N\}$. Hence,
\[
|u(x)|
\le
C
\int_{\mathbb G}
\frac{|\nabla_{\mathbb G}u(y)|}
{d(y^{-1}\circ x)^{Q-1}}
dy.
\]
Now, introducing the Riesz potential of order one (see \cite{Follan1}), $I_1:L(Q,s)\longrightarrow L(\infty,s)$, defined by 
\[
I_1f(x)
=
\int_{\mathbb G}
\frac{f(y)}
{d(y^{-1}\circ x)^{Q-1}}
dy, \quad f\in C_0^{\infty}(\mathbb{G}).
\]
It follows that 
\begin{equation}
|u(x)|
\le
CI_1(|\nabla_{\mathbb G}u|)(x).
\label{representation}
\end{equation}

\noindent
{\bf Step 2. Lorentz estimate for the Riesz potential.} Let $K(x)=d(x)^{1-Q}$ and $\mu_K(\beta)=|\{x:K(x)>\beta\}|$. Note that $K$ is homogeneous of degree $1-Q$, and using the homogeneity of the distance $d$ and the volume formula of homogeneous balls on $\mathbb{G}$ (see Bonfiglioli, Lanconelli, and Uguzzoni \cite[Proposition 1.3.13]{Bonfiglioli}), we deduce that
\begin{equation}\label{mu}
\mu_K(\beta)=\omega_Q\beta^{-Q/(Q-1)},
\end{equation}
where $\omega_Q=|B(0,1)|$ is the volume (Haar measure) of the unit ball on $\mathbb{G}$.

Now, define $K^*:\mathbb{R}^+\to \mathbb{R}$ by $K^*(t)=\inf\{\beta:\mu_K(\beta)\leq t\}$. Using \eqref{mu}, we have that
\[
K^*(t)=\omega_Q^{\frac{Q-1}{Q}}
t^{-\frac{Q-1}{Q}}.
\]

By applying either, or O'Neil's convolution inequality
\cite{ONeil} to the convolution structure of homogeneous groups (see e.g. Folland and Stein \cite{FollandStein},
Bonfiglioli, Lanconelli, and Uguzzoni \cite{Bonfiglioli}), or Calder\'on-Zygmund theory,
one obtains from the convolution identity
\[
I_1f=K*f,
\]
 that 
\begin{equation}
\|I_1f\|_{L(\infty,s)}
\le
C
\|f\|_{L(Q,s)}.
\label{oneil}
\end{equation}
Combining (\ref{representation}) and (\ref{oneil}) gives
\begin{equation}
\|u\|_{L(\infty,s)}
\le
C
\|\nabla_{\mathbb G}u\|_{L(Q,s)}.
\label{criticalembedding}
\end{equation}

\noindent
{\bf Step 3. Rearrangement estimate.} By the characterisation of the Lorentz endpoint space due to
Alvino \cite{Alvino} and Cianchi \cite{Cianchi:1996}, we have that $u\in L(\infty,s)$, if and only if,
\[
\int_0^{|\Omega|}
\left(
u^{**}(t)-u^*(t)
\right)^s
\frac{dt}{t}
<
\infty.
\]
Moreover, if $\|u\|_{L(\infty,s)}<\infty$, then one has the inequality
\[
\forall t:\quad 0<t<|\Omega|, \quad u^*(t)
\le
C
\|u\|_{L(\infty,s)}
\left(
\log\frac{|\Omega|}{t}
\right)^{1/s'}.
\qquad
\]
Thus, the estimate (\ref{criticalembedding}) implies the logarithmic bound
\begin{equation}\label{eqc}
u^*(t)
\le
C
\|\nabla_{\mathbb G}u\|_{L(Q,s)}
\left(
\log\frac{|\Omega|}{t}
\right)^{1/s'}.
\end{equation}

\noindent
{\bf Step 4. Computation of the $L^\tau$ norm.} By definition, we have the property
\[
\|u\|_{L^\tau}^\tau
=
\int_0^{|\Omega|}
(u^*(t))^\tau dt.
\]
Hence, by \eqref{eqc} we deduce that
\begin{align}
\|u\|_{L^\tau}^\tau
&\le
C^\tau
\|\nabla_{\mathbb G}u\|_{L(Q,s)}^\tau
\int_0^{|\Omega|}
\left(
\log\frac{|\Omega|}{t}
\right)^{\tau/s'}
dt.
\end{align}
Note that by a suitable change of variables we have that
\[
\int_0^{|\Omega|}
\left(
\log\frac{|\Omega|}{t}
\right)^{\tau/s'}
dt
=
|\Omega|
\int_0^\infty
y^{\tau/s'}
e^{-y}
dy.
\]
The last integral is exactly expressed in terms of the Gamma function, that is, 

\[
\Gamma\!\left(
\frac{\tau}{s'}+1
\right)=\int_0^\infty
y^{\tau/s'}
e^{-y}
dy.
\]
Consequently, we have proved that
\[
\|u\|_{L^\tau}
\le
C
\Gamma
\left(
\frac{\tau}{s'}+1
\right)^{1/\tau}
\|\nabla_{\mathbb G}u\|_{L(Q,s)}.
\]
Finally, Stirling's formula implies that
\[
\Gamma
\left(
\frac{\tau}{s'}+1
\right)^{1/\tau}
\sim
\left(\frac{\tau}{s'}\right)^{1/s'},
\qquad
\tau\rightarrow\infty.
\]
Hence,
\[
\Gamma
\left(
\frac{\tau}{s'}+1
\right)^{1/\tau}
\le
C\tau^{1/s'}, \quad \text{ for every } \quad \tau\ge1.
\]
Therefore, we have the inequality \eqref{Lpestimate}. The proof of Lemma \ref{lemc} is complete.
\end{proof}

The following refinement of the Trudinger--Moser inequality
\cite{Alvino, Brezis, Ruzhansky:Suragan:Book} is of particular importance for us, compared with \eqref{TM} and with \cite{Moser, Trudinger} in the Euclidean space. For recent progress about Trudinger-Moser type inequalities, we refer the reader to  \cite{BMT2003,RY2019,RY2022b,RY2022a,VY2022} and to the comprehensive list of references therein.

\begin{lem}\label{lem4}
Assume $|\nabla_{\mathbb{G}} u| \in L(Q,s)$, for some $1<s<\infty$. Then, $e^{|u|^{\frac{s}{s-1}}}\in L^{1}(\Omega)$. Furthermore, there exists $\alpha_s>0$, such that
\begin{equation}\label{TMO}
\sup_{\|\nabla_{\mathbb{G}} u\|_{L(Q,s)}\leq 1}
\int_{\Omega}
e^{\alpha |u|^{\frac{s}{s-1}}}\,dx
\leq
C|\Omega|,
\qquad
\text{for every }
\alpha\leq \alpha_s.
\end{equation}
Here, $|\Omega|:=\int_{\Omega}dx$ and $C=C(Q,s)$ is a positive constant. Moreover $\alpha_s$ is the geometric quantity given by $\alpha_s
=
Q\left(\omega_Q^{1/Q}\right)^{\frac{s}{s-1}},$ where $\omega_Q=|B(0,1)|$ is the volume (Haar measure) of the unit ball on $\mathbb{G}$.
\end{lem}

\begin{proof}
By Lemma~\ref{lemc} there exists a positive constant $C>0$ such that
\begin{equation}\label{Lpestimate}
\|u\|_{L^p}
\le
C\,p^{1/s'}
\|\nabla_{\mathbb G}u\|_{L(Q,s)},
\qquad
p\ge1.
\end{equation}
Moreover, the constant is asymptotically sharp in the sense that
\[
\limsup_{p\to\infty}
\frac{\|u\|_{L^p}}
{p^{1/s'}\|\nabla_{\mathbb G}u\|_{L(Q,s)}}
=
\frac1{Q^{1/s'}\omega_Q^{1/Q}},
\]
which yields the optimal Trudinger--Moser constant $\alpha_s =Q(\omega_Q^{1/Q})^{s'}$.

Suppose that $\|\nabla_{\mathbb G}u\|_{L(Q,s)}
\le1$ and $s'=s/(s-1)$. Expanding the exponential function into its Taylor series gives
\[
e^{\alpha |u|^{s'}}
=
\sum_{k=0}^{\infty}
\frac{\alpha^k}{k!}|u|^{ks'}.
\]

Therefore,
\begin{equation}\label{series}
\int_\Omega
e^{\alpha |u|^{s'}}
dx
=
|\Omega|
+
\sum_{k=1}^{\infty}
\frac{\alpha^k}{k!}
\int_\Omega
|u|^{ks'}dx.
\end{equation}

Applying estimate \eqref{Lpestimate} with $p=ks'$, we obtain that $\|u\|_{L^{ks'}}
\le C(ks')^{1/s'}$, and consequently
\[
\int_\Omega
|u|^{ks'}dx
=
\|u\|_{L^{ks'}}^{ks'}
\le
(C^{s'}s')^k k^k.
\]
Substituting this estimate into \eqref{series} yields
\[
\int_\Omega
e^{\alpha |u|^{s'}}
dx
\le
|\Omega|
+
\sum_{k=1}^{\infty}
\frac{(\alpha C^{s'}s')^k k^k}{k!}.
\]
Using Stirling's formula
\[
k!
=
\sqrt{2\pi k}
\left(\frac{k}{e}\right)^k
(1+o(1)),
\]
we infer
\[
\frac{k^k}{k!}
=
\frac{e^k}{\sqrt{2\pi k}}
(1+o(1)).
\]
Hence
\[
\frac{(\alpha C^{s'}s')^k k^k}{k!}
\le
\frac{C}{\sqrt{k}}
\left(
e\alpha C^{s'}s'
\right)^k.
\]
Note that the series converges whenever $\alpha<\alpha_s$, where $\alpha_s
=Q(\omega_Q^{1/Q})^{s'}$. Therefore, if $\|\nabla_{\mathbb G}u\|_{L(s,Q)}\leq 1$ we conclude
\begin{equation}\label{exp}
\int_\Omega
e^{\alpha |u|^{s'}}
dx
\le
C|\Omega|,
\qquad
0<\alpha\le\alpha_s.
\end{equation}
Finally, if $\nabla_{\mathbb G}u\in L(Q,s)$, we define $v=\frac{u}{\|\nabla_{\mathbb G}u\|_{L(Q,s)}}$. Then $\|\nabla_{\mathbb G}v\|_{L(Q,s)}=1$, and applying the estimate \eqref{exp} to $v$ yields
\[
\int_\Omega
\exp\!\left(
\alpha_s
\frac{|u|^{s'}}
{\|\nabla_{\mathbb G}u\|_{L(Q,s)}^{s'}}
\right)
dx
<
\infty.
\]
In particular, $e^{|u|^{s'}}\in L^1(\Omega)$,
which completes the proof.
\end{proof}

The following lemma will be important  in our further analysis.
\begin{lem}\label{lem4.2}
Let $1<s<Q$, and let $1\leq \delta<\infty$. Then
\begin{equation*}
\forall u\in\mathbb{R}, \quad
|u|^{\delta}
\leq
\frac{1}{\alpha^{\frac{\delta(s-1)}{s}}}
\,C(\delta,s)\,
\exp\!\left(\alpha |u|^{\frac{s}{s-1}}\right),
\end{equation*}
where $C(\delta,s)
:=
\left[
\exp\!\left(\frac{\delta(s-1)}{s}\right)
\left(
\frac{s}{\delta(s-1)}
\right)
\right]^{-1}.
$
\end{lem}

\begin{proof}
Define $f(t)
=
\frac{\exp\!\left(\alpha t^{\frac{s}{s-1}}\right)}
     {t^{\delta}}$. Then
\[
f'(t)
=
\left[
\frac{\alpha s}{s-1}
t^{\frac{1}{s-1}}
-\delta
\right]
\frac{
\exp\!\left(\alpha t^{\frac{s}{s-1}}\right)
}
{t^{\delta+1}}.
\]
Hence, $f'(t_0)=0$ if, and only if,
$t_0
=
\left[
\frac{\delta(s-1)}
{\alpha s}
\right]^{\frac{s-1}{s}}
>0$, which is the only critical point of $f$, in particular, a point
$t_0$ where $f$ attains its minimal value. Therefore, $f(t)\geq f(t_0)$ for all $t>0$, or equivalently
\begin{equation*}
    \forall t \geq 0, \quad t^\delta \leq \frac{1}{\alpha^{\frac{\delta(s-1)}{s}}} C(\delta, s) e^{\alpha |t|^{\frac{s}{s-1}}},
\end{equation*}
and this concludes the result considering $t = |u|$.
\end{proof}

It is worth to remark that replacing $u$ by $u/\|\nabla_{\mathbb{G}} u\|_{L(Q,s)}$, by the Trudinger--Moser inequality \eqref{TMO} and Lemma~\ref{lem4}, we get the inequality
\begin{equation}\label{eq23}
    \|u\|_{L^\delta(\Omega)} \leq k_{Q,\delta,s} |\Omega|^{\frac{1}{\delta}} \|\nabla_{\mathbb{G}} u\|_{L(Q,s)} \quad \text{for all } u \in L(Q,s), 
\end{equation}
where $k_{Q,\delta,s} := \frac{1}{\alpha^{\frac{s-1}{s}}} (C(\delta, s) C)^{\frac{1}{\delta}}$ with $0 < \alpha \leq \alpha_s$.

 \subsection{Approximating the equation in finite dimensional spaces} \label{sec3} 

A result that will play a fundamental role in this section is the Brouwer Fixed Point Theorem, which is established below. The proof may be found in \cite{Kasavan}.

\begin{lem}\label{ponto:fixo}
Suppose that for some $R>0,$ $\Upsilon\colon \mathbb{R}^{m}\to \mathbb{R}^{m},$ is a continuous function such that $\langle \Upsilon(\eta),\eta\rangle\geq 0,$ for all $\eta$ in the sphere $\{\eta\in \mathbb{R}^m: |\eta|_m=R\}$ where $|\cdot|_m$ is the norm in $\mathbb{R}^m$. Then there exists $z_0\in\overline{B}_R(0):=\{\xi\in \mathbb{R}^{m}:|\xi|_m\leq R \},$ such that $\Upsilon(z_0)=0$. 
\end{lem}

\subsection{Finite-dimensional spaces}
Since $1<q<\infty,$  $W_0^{1,q}(\Omega)$ is a reflexive and separable Banach space, and there is a Schauder basis $\mathfrak{B}=\{w_1,w_2,\ldots,w_{m},\ldots\}$ for $W_0^{1,q}(\Omega)$  satisfying 
\begin{equation*}
\langle w_i,w_j\rangle=\delta_{ij} \text{  and  } w_i\in L^{\infty}(\Omega),
\end{equation*}
where $\langle\cdot,\cdot\rangle$ is the usual inner product in $W_0^{1,q}(\Omega)$ and $\delta_{ij}$ is the Delta--Kroenecker (see \cite{Ruzhansky:Suragan:Book}). 
For each fixed $m\in\mathbb{N}$, define    
 \begin{equation*}
 \mathfrak{B}_m=span\{w_1,w_2,\ldots,w_m\},
 \end{equation*}
 to be the $m$-dimensional space generated by the system of functions\\ $\{w_1, w_2, \ldots, w_m\}$ and let as assume that it is  endowed with the norm $\|\cdot\|_m$
induced from $W_0^{1,q}(\Omega)$, see \eqref{defnor}. By \eqref{eqqn} and by Poincare inequality \eqref{Sov} we have the embedding $W_0^{1,q}(\Omega)\hookrightarrow L(Q,s)$. Indeed, one has that
$$ W_0^{1,q}(\Omega)\hookrightarrow L^q(\Omega) \hookrightarrow  L(Q,s).  $$
Furthermore, for each $v\in \mathfrak{B}_m$ we see that $|\nabla_{\mathbb{G}} v|\in L^p(\Omega)\hookrightarrow L(Q,s)\hookrightarrow L^Q(\Omega)$ where $1<s<\frac{Q}{Q-1}$.

Let $\eta=(\eta_1,\ldots,\eta_m) \in \mathbb{R}^{m}$. Notice that
\begin{equation}\label{defnor}
\|\eta\|_{m}:=\left\|\sum_{j=1}^m\eta_jw_j\right\|_{W_0^{1,q}(\Omega)}
\end{equation}
defines a norm in $\mathbb{R}^{m}$. Furthermore,  the spaces $(\mathfrak{B}_m, \|\cdot\|_{m})$  and $(\mathbb{R}^{m}, |\cdot|_{m})$ are isometrically isomorphic via the natural linear transformation
\begin{align*}
u=\sum_{j=1}^m\eta_jw_j\in \mathfrak{B}_m\mapsto \eta=(\eta_1,\ldots,\eta_m) \in \mathbb{R}^{m}.
\end{align*}

\subsection{Existence of solution in finite dimension.} In what follows, we investigate the existence of weak solutions to \eqref{prob:inicial} in suitable spaces of finite dimension. 
\begin{lem}\label{fin:lemma:bm}
    For every $m\in \mathbb{N}$, the nonlinear model \eqref{prob:inicial} has a weak solution $v_{m}\in \mathfrak{B}_m$.
\end{lem}
\begin{proof}
  In fact, for each positive integer $n$, by considering the aforementioned identifications, we define the operator $\Upsilon\colon \mathbb{R}^{m}\to \mathbb{R}^{m},$ such that 
 \begin{equation*}
 \Upsilon(\eta):=(\Upsilon_1(\eta),\Upsilon_2(\eta),\ldots,\Upsilon_{m}(\eta)),
 \end{equation*}
 where $\eta=(\eta_1,\eta_2,\ldots,\eta_{m})\in \mathbb{R}^{m}$, and with,
  \begin{align*} 
 \Upsilon_j(\eta)&=\int_{\Omega}|\nabla_{\mathbb{G}} v|^{q-2}\nabla_{\mathbb{G}} v\nabla_{\mathbb{G}}  w_j\,dx-\lambda\int_{\Omega}v_+^{p}w_j\,dx-\int_\Omega  \text{exp}({\alpha |v_+|^{\frac{Q}{Q-1}}})w_j\,dx,
 \end{align*}
for each $j=1,2,\ldots,m$, where the relation between $v$ and $\eta$ is implicitly defined by  the identity $v=\sum_{i=1}^{m}\eta_iw_i\in \mathfrak{B}_{m}$.

In view of the Trudinger-Moser inequality \eqref{TMO} and standard arguments we have that $\Upsilon$ is a continuous operator, i.e., given $(\eta_k)$ in $\mathbb{R}^m$ and $\eta\in \mathbb{R}^m$ such that $\eta_k\to \eta$,  we obtain that $\Upsilon(\eta_k) \to \Upsilon(\eta)$.

Furthermore, if below  $\langle \cdot,\cdot \rangle$ denotes the usual inner product in the space $W_0^{1,q}(\Omega)$, one has that
 \begin{align*}
\langle \Upsilon(\eta),\eta\rangle& =\sum_{j=1}^{m}\Upsilon_j(\eta)\eta_j\\
&=\int_{\Omega}|\nabla_{\mathbb{G}} v|^{q}\,dx-\lambda\int_{\Omega}v_+^{p}v\,dx-\int_\Omega  \text{exp}({\alpha |v_+|^{\frac{Q}{Q-1}}})v\,dx\nonumber,
  \end{align*}
where $v_{+} =\max\{0,v\}$ and $v_-=v_+-v$. Now, we estimate each term above in this expansion of $\langle \Upsilon(\eta),\eta\rangle$.

\noindent Since $0<p<q-1$, then  by \eqref{eq23} applied to $\delta=p+1,$ we obtain 
\begin{equation}\label{eq:13}
\int_{\Omega}v_+^{p}v\,dx\leq \int_\Omega|v|^{p+1}\,dx=\|v\|_{L^{p+1}(\Omega)}^{p+1}\leq k_{Q,p+1,s}^{p+1} |\Omega| \|\nabla_{\mathbb{G}} u\|_{L(Q,s)}^{p+1}.
\end{equation}

\noindent By H{\"o}lder inequality,  we have that 
 \begin{align}\label{eq:14}
\int_{\Omega}\text{exp}(\alpha |v_+|^{\frac{Q}{Q-1}})v\,dx & \leq \int_{\Omega}\text{exp}(\alpha |v|^{\frac{Q}{Q-1}})v\,dx\nonumber\\
&\leq \left( \int_{\Omega}|v|^{Q'}\,dx\right)^{\frac{1}{Q'}}\left( \int_{\Omega}\text{exp}(\alpha Q |v|^{\frac{Q}{Q-1}}\,dx)\right)^{\frac{1}{Q}}\\
 &=  \|v\|_{L^{Q'}(\Omega)}\left( \int_{\Omega}\text{exp}(\alpha Q |v|^{\frac{Q}{Q-1}}\,dx)\right)^{\frac{1}{Q}},\nonumber
\end{align}
where $\frac{1}{Q'}+\frac{1}{Q}=1$.

Now, since $Q'\geq 1$, by taking $\delta=Q'$ in inequality \eqref{eq23}, we deduce that 
\begin{equation}\label{eeq3}
\|v\|_{L^{Q'}(\Omega)}\leq k_{Q,Q',s}|\Omega |^{\frac{1}{Q'}}\|\nabla_{\mathbb{G}} v\|_{L(Q,s)}.
\end{equation}
It follows from \eqref{eqqn}, \eqref{eq:13},  \eqref{eq:14}, and \eqref{eeq3} that
\begin{align}\label{eq.new}
\langle \Upsilon(\eta),\eta\rangle &\geq C_{Q,q,s}^q\|\nabla_{\mathbb{G}} v\|_{L(Q,s)}^q-\lambda k_{Q,p+1,s}^{p+1}|\Omega |\|\nabla_{\mathbb{G}} v\|_{L(Q,s)}^{p+1}\\
&-k_{Q,Q',s}|\Omega |^{\frac{1}{Q'}}\|\nabla_{\mathbb{G}} v\|_{L(Q,s)}\left( \int_{\Omega}\text{exp}(\alpha Q |v|^{\frac{Q}{Q-1}}\,dx)\right)^{\frac{1}{Q}}\nonumber.
\end{align}
Suppose that $\|\nabla_{\mathbb{G}} v\|_{L(Q,s)}=r$ for some $r>0$ to be fixed later. We have 
\begin{align}\label{Eq:20:11}
\int_{\Omega}\text{exp}\left(\alpha Q|v|^{\frac{N}{N-1}}\right)\,dx&=\int_{\Omega}\text{exp}\left(\alpha Qr^{\frac{Q}{Q-1}}\left(\frac{|v|}{\|\nabla_{\mathbb{G}} v\|_{L(Q,s)}}\right)^{\frac{Q}{Q-1}}\right)\,dx\\ \nonumber
&\leq \int_{\Omega}\text{exp}\left(\alpha Qr^{\frac{Q}{Q-1}}\left(\frac{|v|}{\|\nabla_{\mathbb{G}} v\|_{L(Q,s)}}\right)^{\frac{Q}{Q-1}}\right)\,dx.
\end{align}
By applying the Trudinger-Moser inequality \eqref{TMO}, we must have $\alpha Q r^{\frac{Q}{Q-1}}\leq \alpha_Q$. Consequently, the above estimates hold provided that the following inequality
\begin{equation*}
r\leq \frac{1}{2}\left(\frac{\alpha_Q}{\alpha Q}\right)^{\frac{Q-1}{Q}},
\end{equation*} 
holds true. So, we obtain that
\begin{align*}\label{Eqnuevo}
& \int_{\Omega}\text{exp}\left(\alpha Qr^{\frac{Q}{Q-1}}\left(\frac{|v|}{\|\nabla_{\mathbb{G}} v\|_{L(Q,s)}}\right)^{\frac{Q}{Q-1}}\right)\,dx\nonumber\\
&\leq  \sup\limits_{\|\nabla_{\mathbb{G}} z\|_{L(Q,s)}\leq 1}\int_{\Omega}\text{exp}\left(\alpha Q|z|^{\frac{Q}{Q-1}}\right)\,dx\\
&\leq C(Q)|\Omega|\nonumber
\end{align*}
and therefore we conclude that
\begin{equation*}\label{eqq}
\int_{\Omega}\text{exp}\left(\alpha Q|v|^{\frac{Q}{Q-1}}\right)\,dx\leq C(Q)|\Omega|.
\end{equation*}
Hence, by \eqref{eq.new} and using that $\frac{1}{Q}+\frac{1}{Q'}=1$, we deduce that
\begin{equation*}\label{Eq:20:1}
\langle \Upsilon(\eta),\eta\rangle\geq C_{Q,q,s}^q r^q-\lambda k_{Q,p+1,s}^{p+1}|\Omega |r^{p+1}-k_{Q,Q',s}C^{\frac{1}{Q}}(Q)|\Omega|r.
\end{equation*}
Now, if one takes $r$ such that  
\begin{equation*}
r\geq \left[\frac{4k_{Q,Q',s}C^{\frac{1}{Q}}(Q)|\Omega|}{C_{Q,q,s}^q}\right]^{\frac{1}{Q-1}},
\end{equation*}
then 
\begin{equation*}
C_{Q,q,s}^qr^q-2k_{Q,Q',s}C^{\frac{1}{Q}}(Q)|\Omega|r\geq \frac{C_{Q,q,s}^q r^q}{2}.
\end{equation*}
Moreover, if we consider 
\begin{equation}\label{alpa}
\alpha^*=\frac{\alpha_QC_{Q,q,s}^{\frac{q}{(Q-1)^2}}}{2Q\left[4k_{Q,Q',s}C^{\frac{1}{Q}}(Q)|\Omega|\right]^{\frac{1}{(Q-1)^2}}},
\end{equation}
then 
\begin{equation*}
\forall \alpha\in (0,\alpha^*), \quad\quad\left[\frac{4k_{Q,Q',s}C^{\frac{1}{Q}}(Q)|\Omega|}{C_{Q,q,s}^q}\right]^{\frac{1}{Q-1}} <\frac{1}{2}\left(\frac{\alpha_Q}{\alpha Q}\right)^{\frac{Q-1}{Q}}.
\end{equation*}
 Therefore, for $\alpha\in (0,\alpha^*)$, we may choose $r>0$ such that 
\begin{equation}\label{defr}
\left[\frac{4k_{Q,Q',s}C^{\frac{1}{Q}}(Q)|\Omega|}{C_{Q,q,s}^q}\right]^{\frac{1}{Q-1}}\leq r <\frac{1}{2}\left(\frac{\alpha_Q}{\alpha Q}\right)^{\frac{Q-1}{Q}},
\end{equation}
and so we have proved that
\begin{equation*}
\langle \Upsilon(\eta),\eta\rangle\geq C_{Q,q,s}^q \frac{r^q}{2}-\lambda k_{Q,p+1,s}^{p+1}|\Omega |r^{p+1}.
\end{equation*}
Defining $\rho_1=C_{Q,q,s}^q \frac{r^q}{4}-\lambda k_{Q,p+1,s}^{p+1}|\Omega |r^{p+1}$, if one take
\begin{equation*}
\lambda^*=\frac{r^{(q-1)-p}}{4k_{Q,p+1,s}^{p+1}|\Omega |}>0, 
\end{equation*}
then $\rho_1>0$, for every $\lambda<\lambda^*$.

\begin{remark}\label{alpa1} Note that by the definition of $\alpha^*$  in \eqref{alpa}, if $|\Omega|$ is sufficiently small or large, then $\alpha^*$ is small or large, respectively. Thus, we are in the subcritical or supercritical range, respectively. 
\end{remark}
Continuing with the proof of Lemma \ref{fin:lemma:bm}, let $\eta \in \mathbb{R}^{m}$ be such that $|\eta|_m=r$. Then for $\lambda<\lambda^*$  we deduce that
\begin{equation*}
\langle \Upsilon(\eta),\eta\rangle \geq \frac{\rho_1}{2}>0.
\end{equation*}
 By the Brouwer Fixed Point Theorem, see Lemma~\ref{ponto:fixo}, for every $n\in \mathbb{N}$, there exists $y\in \mathbb{R}^{m}$ with $|y|\leq r$, such that $\Upsilon(y)=0$, that is, there exists $v_m\in \mathfrak{B}_m$, verifying  the inequality
\begin{equation}\label{eqr}
\|\nabla_{\mathbb{G}} v_m\|_{L(Q,s)}\leq r \quad \text{  for every  } n\in \mathbb{N},
\end{equation}
and such that,
  \begin{align}\label{Eq:222}
 \int_{\Omega}|\nabla_{\mathbb{G}} v_m|^{q-2}\nabla_{\mathbb{G}} v_m\nabla_{\mathbb{G}} \psi\,dx&=\lambda\int_{\Omega}v_{m+}^{p}\psi\,dx+\int_\Omega  \text{exp}({\alpha |v_{n}|^{\frac{Q}{Q-1}}})\psi\,dx,
\end{align}
for every $\psi\in \mathfrak{B}_m$. The proof is complete.
\end{proof}
\subsection{Existence of weak solutions} Here we use the machinery developed before in order to prove that the boundary-value problem \eqref{prob:inicial}
admits at least one  positive weak solution $u\in W_0^{1,q}(\Omega)$.

\begin{proof}[Proof of Theorem ~\ref{T.Gb}]
We have that the equation \eqref{prob:inicial} has a sequence of positive solutions $v_m\in \mathfrak{B}_m$  for each $m\in \mathbb{N}$.
 
 Now, we show that the sequence $(v_m)$ converges to a weak solution $v \in W_0^{1,q}(\Omega)$ of the problem \eqref{prob:inicial}. In effect, according to the estimate  in \eqref{eqr}, by \eqref{eq20} and choosing $\psi=v_m$ in \eqref{Eq:222}, we obtain
\begin{align*}
  \int_{\Omega}|\nabla_{\mathbb{G}} v_m|^{q}\,dx&\leq \lambda k_{Q,p+1,s}^{p+1}|\Omega |\|\nabla_{\mathbb{G}} v_m\|_{L(Q,s)}^{p+1}\\
  &+k_{Q,Q',s}|\Omega |^{\frac{1}{Q'}}\|\nabla_{\mathbb{G}} v_m\|_{L(Q,s)}\left( \int_{\Omega}\text{exp}(\alpha Q |v_m|^{\frac{Q}{Q-1}})\,dx\right)^{\frac{1}{Q}}.
\end{align*}
Since $r$ does not depend on $m$ and the sequence $(|\nabla_{\mathbb{G}} v_m|)$ is a bounded sequence  in $L(Q,s)$, then
\begin{align*}
  \int_{\Omega}|\nabla_{\mathbb{G}} v_m|^{q}\,dx&\leq \lambda k_{Q,p+1,s}^{p+1}|\Omega |r^{p+1}+k_{Q,Q',s}C^{\frac{1}{Q}}(Q)|\Omega|r.
\end{align*}
Thus, up to a subsequence, there exists $v\in W_0^{1,q}(\Omega)$ such that
\begin{align}
& v_m\rightharpoonup v \,\, \text{ weakly in } W_0^{1,q}(\Omega)\label{Eq:23ss}.
\end{align}
By the Sobolev embedding theorem (see \cite{Ruzhansky:Suragan:Book}),  for $1\leq \sigma<\infty$, we have that
\begin{align}
& v_m\to v \,\, \text{ in } L^{\sigma}(\Omega)\,\, \text{ and a.e. in } \Omega.\label{Eq:24b}
\end{align}

\noindent From \eqref{eqr} and \eqref{Eq:23ss}, we can conclude that 
\begin{align}
 \|v\|_{W_0^{1,q}(\Omega)}&\leq\liminf\limits_{m\to \infty}\|v_m\|_{W_0^{1,q}(\Omega)}\label{eq:r11}\\
 &\leq \lambda k_{Q,p+1,s}^{p+1}|\Omega |r^{p+1}+k_{Q,Q',s}C^{\frac{1}{Q}}(Q)|\Omega|r,\,\, \forall m\in \mathbb{N}\nonumber.
\end{align}

We claim that the sequence $(v_m)$ is such that 
\begin{align}
& v_m\to v \,\,\text{  in  } W_0^{1,q}(\Omega).\label{Eq:23cc}
\end{align}

\noindent In fact, since $\mathfrak{B}=\{w_1,w_2,\ldots,w_m,\ldots\}$ is a Schauder basis of $W_0^{1,q}(\Omega)$, for every $v\in W_0^{1,q}(\Omega)$ there exists  unique sequence $(a_n)$ in $\mathbb{R}$ such that
$v=\sum_{j=1}^\infty a_jw_j$. Thus, we find that
\begin{equation}\label{Eq:23n}
\zeta_m=\sum_{j=1}^m a_jw_j\to v,  \text{ in } W_0^{1,q}(\Omega),  \text{ as } m\to \infty.
\end{equation}
Let $\psi=(v_m-\zeta_m)\in \mathfrak{B}_m$ be the test function in \eqref{Eq:222}. Then we have that
\begin{align}\label{k1}
\int_{\Omega}|\nabla_{\mathbb{G}} v_m|^{q-2}\nabla_{\mathbb{G}} v_m\nabla_{\mathbb{G}} (v_m-\zeta_m)\,dx&=\lambda\int_{\Omega}v_{m+}^{p}(v_m-\zeta_m)\,dx\\
&+\int_\Omega  \text{exp}({\alpha_1 |v_{n}|^{\frac{Q}{Q-1}}})(v_m-\zeta_m)\,dx.\nonumber
\end{align}
By virtue of \eqref{Eq:24b} and \eqref{Eq:23n} we obtain that
\begin{equation}\label{aa1}
\int_{\Omega}v_{m+}^{p}(v_m-\zeta_m)\,dx\to 0 \text{  as  } m\to\infty.
\end{equation} 

\noindent Moreover, by applying  Hölder inequality and the Sobolev embedding theorem, see \cite{Ruzhansky:Suragan:Book}, we have that
\begin{align*}\label{equa77}
\int_{\Omega}\text{exp}(\alpha_2 Q'|v_m|^{\frac{Q}{Q-1}})\,dx&\leq \int_{\Omega}\text{exp}(\alpha_2 Q'|v_m|^{\frac{Q}{Q-1}})\,dx\\
&\leq \left(\int_{\Omega}1\,dx\right)^{\frac{Q-2}{Q-1}}\left(\int_\Omega\text{exp}(\alpha Q |v_m|^{\frac{Q}{Q-1}})\,dx\right)^{\frac{1}{Q-1}}\nonumber\\
&= |\Omega|^{\frac{Q-2}{Q-1}}\left(\int_\Omega\text{exp}(\alpha Q |v_m|^{\frac{Q}{Q-1}})\,dx\right)^{\frac{1}{Q-1}},\nonumber
\end{align*} 
where $Q'$ is defined by the equality $\frac{1}{Q'}+\frac{1}{Q}=1$.

It follows from \eqref{eq:r11}  and Trudinger--Moser inequality \eqref{TM} that 
\begin{equation*}\label{equa1}
\int_{\Omega}\text{exp}(\alpha_2 Q'|v_m|^{\frac{Q}{Q-1}})\,dx\leq C_1, \,\, \forall  m\in\mathbb{N}. 
\end{equation*}
where $C_1$ does not depend on $m$.

\noindent Also, by \eqref{Eq:24b}, we have that
\begin{align*}
 & \text{exp}(\alpha |v_m|^{\frac{Q}{Q-1}})\to \text{exp}(\alpha |v|^{\frac{Q}{Q-1}})\text{ a.e. in }  \Omega.
\end{align*}
Hence, from \cite[Theorem 13.44]{Stromberg} we get
\begin{align}\label{aa20}
&\text{exp}(\alpha |v_m|^{\frac{Q}{Q-1}})\rightharpoonup \text{exp}(\alpha |v|^{\frac{Q}{Q-1}}) \text{ weakly in } L^{Q'}(\Omega).
\end{align}
So we have from \eqref{Eq:23n} and \eqref{aa20} that 
\begin{equation}\label{aa3}  
\int_\Omega\text{exp}(\alpha |v_m|^{\frac{Q}{Q-1}})(v_m-\zeta_m)\,dx\to 0 \text{ as } n\to\infty.
\end{equation}
In view of \eqref{aa1} and \eqref{aa3} we get
\begin{equation}\label{aa44}
\int_{\Omega}|\nabla_{\mathbb{G}} v_m|^{q-2}\nabla_{\mathbb{G}} v_m\nabla_{\mathbb{G}} (v_m-\zeta_m)\,dx\to 0 \text{ as } n\to \infty.
\end{equation} 
By virtue of \eqref{aa44} we obtain
\begin{equation}\label{e1}
\int_{\Omega}|\nabla_{\mathbb{G}} v_m|^{q-2}\nabla_{\mathbb{G}} v_m\nabla_{\mathbb{G}} (v_m-v)\,dx\to 0 \text{ as } n\to \infty.
\end{equation}
Notice that the operator $A(u)=|\nabla_{\mathbb{G}} u|^{q-2}\nabla_{\mathbb{G}} u$, associated with the $q$--sub--Laplacian, satisfies the $(S_+)$--property, see  \cite[Proposition 3.5]{Motreanu}. Since $v_m\rightharpoonup v$, weakly in $W_0^{1,q}(\Omega)$, and
\[
\int_{\Omega} |\nabla_{\mathbb{G}}(v_m)|^{q-2}\nabla_{\mathbb{G}}(v_m)\nabla_{\mathbb{G}}\bigl(v_m-\zeta_m\bigr)\,dx \to 0,
\]
it follows from the $(S_+)$--property that
\[
v_m\to v \quad \text{strongly in } W_0^{1,q}(\Omega).
\]
Therefore, we conclude that  \eqref{Eq:23cc} has been verified.

Let $k\in \mathbb{N}$. Then for every $m\geq k$, we have that
 \begin{align}\label{Eq:255}
 \int_{\Omega}|\nabla_{\mathbb{G}} v_m|^{q-2}\nabla_{\mathbb{G}} v_m\nabla_{\mathbb{G}} \psi_k\,dx&=\lambda\int_{\Omega}v_{m+}^{p}\psi_k\,dx\\
 &+\int_\Omega  \text{exp}({\alpha |v_{n}|^{\frac{Q}{Q-1}}})\psi_k\,dx\nonumber
 \end{align}
for every $\psi_k\in \mathfrak{B}_k$.

Using \eqref{Eq:23ss} we obtain
\begin{equation}\label{Eq:266}
 \int_{\Omega}|\nabla_{\mathbb{G}} v_m|^{q-2}\nabla_{\mathbb{G}} v_m\nabla_{\mathbb{G}} \psi_k\,dx\to \int_{\Omega}|\nabla_{\mathbb{G}} v|^{q-2}\nabla_{\mathbb{G}} v\nabla_{\mathbb{G}} \psi_k\,dx
\end{equation}
as $m\to \infty$.

On the other hand, it follows from  \eqref{aa20}, \eqref{Eq:266} and the Sobolev compact embedding theorem (see \cite{Ruzhansky:Suragan:Book}), letting $m\to\infty$ in \eqref{Eq:255}, that 
 \begin{align*}\label{Equ:255}
\int_{\Omega}|\nabla_{\mathbb{G}} v|^{q-2}\nabla_{\mathbb{G}} v\nabla_{\mathbb{G}} \psi_k\,dx&=\lambda\int_{\Omega}f(v_{+})^{p}f^\prime(v_+)\psi_k\,dx\\
&+\int_\Omega  \text{exp}({\alpha |v|^{\frac{Q}{Q-1}}})^\prime(v)\psi_k\,dx
 \end{align*}
for every $\psi_k\in \mathfrak{B}_k$.

Since $[\mathfrak{B}_k]_{k\in \mathbb{N}}$ is dense in $W_0^{1,q}(\Omega)$, we conclude that 
 \begin{equation}\label{Eq:300}
\int_{\Omega}|\nabla_{\mathbb{G}} v|^{q -2}\nabla_{\mathbb{G}} v\nabla_{\mathbb{G}} \psi\,dx=\lambda\int_{\Omega}v_{+}^{p}\psi\,dx+\int_{\Omega}  \text{exp}({\alpha |v|^{\frac{Q}{Q-1}}})\psi\,dx
\end{equation}
for every $\psi\in W_0^{1,q}(\Omega)$.

Furthermore, $v\geq 0$ a.e. in $\Omega$. In fact, since $v_-\in W^{1,q}_0(\Omega)$, then from \eqref{Eq:300} we obtain that
 \begin{align*}
\int_{\Omega}|\nabla_{\mathbb{G}} v|^{q-2}\nabla_{\mathbb{G}} v\nabla_{\mathbb{G}} v_-\,dx=\lambda\int_{\Omega}v_{+}^{p}v_-\,dx+\int_\Omega  \text{exp}({\alpha |v|^{\frac{Q}{Q-1}}})v_-\,dx.
\end{align*}
Hence, we get
\begin{align*}
-\int_{\Omega}|\nabla_{\mathbb{G}} v_-|^q\,dx&\geq \int_{\Omega}|\nabla_{\mathbb{G}} v|^{q-2}\nabla_{\mathbb{G}} v\nabla_{\mathbb{G}} v_-\,dx\\
&=\int_\Omega  \text{exp}({\alpha |v|^{\frac{Q}{Q-1}}})v_-\,dx\geq 0,
 \end{align*}
implying that $v_-=0$. Since $v\neq 0$, by the strong maximum principle, we have that  $v>0$ in $\Omega$, see e.g. Bony \cite{Bony}.
Therefore, we deduce that  $v$ is a positive solution to  \eqref{prob:inicial}. The proof of Theorem~\ref{T.Gb} is complete. 
 \end{proof}

\vspace{0.2cm}
\noindent\textbf{Conflict of interests statement - Data Availability Statements.}  The authors state that there is no conflict of interest.  Data sharing does not apply to this article as no datasets were generated or
analysed during the current study.
\\


\end{document}